\documentclass[twoside]{article}

\usepackage[accepted]{aistats2019}
\usepackage{graphicx}
\usepackage{amssymb,amsmath,amsthm}
\usepackage{makecell}

\newcommand{\F}{\mathbb{F}_2}
\newcommand{\Pu}{\mathbf{P}_{\text{uni}}}
\newcommand{\Pp}{\mathbf{P}_{\text{plant}}}
\newcommand{\Pone}{\mathbf{P}_{1}}
\newcommand{\Pn}{\mathbf{P}}
\newcommand{\Ppp}{\mathbf{P}_{\text{plant},\pi}}
\newcommand{\Px}{\mathbf{P}_{x^*}}
\newcommand{\Pxs}{\mathbf{P}_{x^*}}
\newcommand{\cPF}{\cP_{\sf FLAT}}
\newcommand{\Pxp}{\mathbf{P}_{x^*,\pi}}
\newcommand{\Pxsp}{\mathbf{P}_{x,\pi}}
\newcommand{\Pxpp}{\mathbf{P}_{x',\pi}}
\newcommand{\cL}{\mathcal{L}}

\newcommand{\eps}{\varepsilon}

\newcommand{\cB}{\mathcal{B}}

\newcommand{\cP}{\mathcal{P}}

\newcommand{\cS}{\mathcal{S}}

\newcommand{\cV}{\mathcal{V}}

\usepackage[round]{natbib}

\newcommand{\E}{\mathbf{E}}

\newcommand{\cN}{\mathcal{N}}

\newtheorem{theorem}{Theorem}

\newtheorem{lemma}[theorem]{Lemma}
\newtheorem{proposition}[theorem]{Proposition}
\newtheorem{remark}[theorem]{Remark}

\begin{document}

%

%

\twocolumn[

\aistatstitle{Detection of Planted Solutions for Flat Satisfiability Problems}

\aistatsauthor{ Quentin Berthet \And Jordan Ellenberg  }

\aistatsaddress{ Statistical Laboratory \\ DPMMS, University of Cambridge \And Department of Mathematics \\ University of Wisconsin, Madison } ]

\begin{abstract}
We study the detection problem of finding planted solutions in random instances of flat satisfiability problems, a generalization of boolean satisfiability formulas. We describe the properties of random instances of flat satisfiability, as well of the optimal rates of detection of the associated hypothesis testing problem. We also study the performance of an algorithmically efficient testing procedure. We introduce a modification of our model, the light planting of solutions, and show that it is as hard as the problem of learning parity with noise. This hints strongly at the difficulty of detecting planted flat satisfiability for a wide class of tests.
\end{abstract}

\section{Introduction}
The rapid growth in many scientific fields of the size of typical datasets, and the increasingly complex models that are studied, have naturally brought forth the notions of statistical and computational complexity in learning theory. For many learning problems, the algorithmic aspect of inference procedures cannot be ignored: it is necessary to consider jointly the difficulties posed by the presence of noise or random errors, and by computational hardness.

The problem of understanding the trade-offs between algorithmic and statistical efficiency, has therefore attracted a lot of interest \cite{ChaJor13,BerCha16,BerPer17,FonBerPer19}. A particularly successful approach has been to investigate the links between learning problems that naturally arise, inspired by applications, and more abstract problems related to random discrete structures, that have been extensively studied in theoretical computer science. An hypothesis of \cite{Fei02}, based on the hardness of refuting satisfiability in random satisfiability formulas - initially used to prove hardness of approximation for several problems - has been used as a primitive to show hardness of improper learning \cite{DanLinSha12,DanLinSha13a,LivShaSha14}. An hypothesis on the planted clique problem has also been used as a primitive to prove computational limits to inference, initially for sparse principal component detection in \cite{BerRig13a,BerRig13b}, and subsequently for other problems in high dimensional statistics \cite{MaWu13,Che13,WanBerSam16,GaoMaZho15,CaiLiaRak15,WanBerPla16,BalBer18}. 

The desire to understand barriers to learning that come from randomness and computation has naturally brought attention to such fundamental problems, and the questions of learning distributions of their instances, in a computationally efficient manner. Examples include \cite{FelGriRey13,FelPerVem13,FelKot14,FelPerVem14}, investigating  the query complexity of statistical algorithms for these problems \cite{Kea98}, or \cite{Ber15} treating the problem of satisfiability detection as an hypothesis testing problem.  

We consider here a learning problem on sets of flats in $\F^n$, shown to be a generalization of the $k$-{\sf SAT} problem in $n$ variables. We introduce the $k$-{\sf FLAT} problem over sets of $m$ flats of dimension $n-k$, that are flat satisfiable if they do not cover all of $\F^n$. This is analogous to satisfiability formulas, that are satisfiable if the $m$ clauses do not exclude all the assignments. We also introduce a learning problem over these instances. It is formulated as a high-dimensional hypothesis testing problem

We study the optimal rate of detection for this problem, in a minimax sense, based on various parameters. We show that the optimal sample size $m$ scales linearly with the dimension $n$. These rates, derived only using information-theoretic limits, are useful as benchmarks. They give a context to the performance of candidate algorithms, and let us see if there is a gap between what we are able to achieve in a computationally efficient manner and the best possible case. We introduce a polynomial-time algorithm for a test, inspired by a technique of \cite{AroGe11}, and show that the test is successful for a sample of order $n^k$. 

We discuss further the algorithmic aspects of this problem, for different types. An important tool to do so is the introduction of a modification of the problem, denoted by lightly planted flat satisfiability, where for every sample one might ``forget'' to plant a solution, and draw instead from the uniform distribution. This change does not significantly alter the statistical aspects but affects the computational aspects, making it as hard as the ``Learning Parity with Noise'' problem. We also show how this result shows that a wide class of testing methods (including those based on so-called {\em statistical algorithms}) cannot be used for detection of planted solutions for flat satisfiability. Indeed, these procedures are by nature not sensitive to this modification, and view these two problems as equally hard.

These results aim to contribute to a larger discussion on the notion of learning under computational constraints. We provide here an example of a problem where an algorithmically efficient testing method is powerful, given a reasonable - albeit suboptimal - sample size (polynomial in the dimension instead of simply linear). This method is not robust to some modification in the model, where the planted assignment is only almost flat satisfiable. This in turn shows that it is impossible for any procedure that is robust to this modification to be both computationally and statistically efficient. 

This concept of ``weaker planting models'', that do not fundamentally change the statistical nature but make them computationally harder have recently attracted interest (see, e.g. \cite{AwaCharLai15} about hypothesis on planted cliques or dense subgraphs). By design, these generalizations prevent the use of brittle properties of the alternative distributions (the existence of a clique in a random graph, or here of an assignment that satisfies {\em all} clauses) to solve these decision problems. Here, we show that such a modification makes the problem significantly harder for computationally efficient methods. Furthermore, results about this auxiliary problem can be used to establish lower bounds for the original problem, for any method that does not depend on these brittle properties. This could be a useful approach to derive similar results for other problems, and to guide us in understanding which properties of certain distributions can be used by efficient algorithms.

The $k$-{\sf FLAT} problem, and the associated detection problem, are described in Section~\ref{SEC:probdesc}. In Section~\ref{SEC:threshold}, we show that there exists a sharp phase transition for flat satisfiability of random instances, with a an explicit threshold in the linear regime $m=\Delta n$. In Section~\ref{SEC:detect}, we derive the optimal rate of detection, with an optimal constant, that coincides with the flat satisfiability transition. In Section~\ref{SEC:poly}, we show that a test that can be computed in polynomial time will be successful with a sample size that is polynomial in $n$. We introduce and analyz in Section~\ref{SEC:light} the problem of detecting a lightly planted solution. We discuss computational aspects in Section~\ref{SEC:hard}. All proofs are in the appendix.

\section{Problem description}
\label{SEC:probdesc}
\subsection{The $k$-{\sf FLAT} problem }
Consider $\F^n$, the $n$-dimensional coordinate space on $\F$. We are given $V=(V_1,\ldots,V_m)$, a collection of $m$ flats of dimension $n-k$, or $k$-flats on $\F^n$. We denote by $k$-{\sf FLAT} the problem of determining whether there exists an element $x \in \F^n$ that is {\it flat satisfying}, i.e. that does not lie on any of the $V_j$, or alternatively, whether  $\F^n = \cup_j V_j$. We can define the flats by taking $k$ linearly independent linear forms $\ell_{j,1},\ldots,\ell_{j,k}$ and $k$ values $\eps_{j,1},\ldots,\eps_{j,k} \in \F$, and having
$$V_j = \{x\in\F^n \, :\, \ell_{j,i}(x) = \eps_{j,i}\, , \, \forall  i \in [k] \} \, .$$

We note that there are many such descriptions for any flat, but choosing the $\ell_{j,i}$ and $\eps_{j,i}$ uniformly at random does yield the uniform distribution on flats.  We also note that if we constrain the flats to be coordinate-aligned by taking each linear form among the projections on one of the $e_i$s, the $V_j$ can be interpreted as satisfiability clauses on $k$ literals, and the set $V_1,\ldots,V_m$ a satisfiability formula with $m$ clauses: For each $x \in \F^n$, $x$ satisfies the $j$-th clause if and only if $x \notin V_j$, and satisfies the formula if and only if it the case for all the $V_j$. The set of flat satisfying assignments is therefore $\F^n \setminus \cup_j V_j$. The problem described above is therefore a generalization of $k$ satisfiability. Thus, the $k$-{\sf FLAT} problem is {\sf NP}-complete for $k \ge 3$.\\

We denote by $\cS(V)$ the set of flat satisfying elements $\F^n \setminus \cup_j V_j$, and by $Z(V)$ its cardinality. We write $\cS$ and $Z$ when it is not ambiguous. We denote by {\sf FLAT} the set of $V$ that are flat satisfiable, i.e. for which there exists a satisfying element. We will consider asymptotics in the linear regime of $m= \Delta n$, for a constant $\Delta>0$, and $m,n \rightarrow +\infty$.

\subsection{Detection of planted flat-satisfiable assignment}

Given a random instance $V$, our goal is to distinguish two hypotheses for its underlying joint distribution. This detection problem is a generalization of the problem of detecting planted satisfiability \cite{Ber15}. Under the uniform distribution (denoted by $\Pu$) the $V_j$s are independent and identically distributed. Their distribution is uniform on the set of flats of dimension $n-k$. A possible way to generate them is to draw uniformly $k$ linearly independent linear forms $\ell_{j,1},\ldots,\ell_{j,k}$ and independently $k$ values $\eps_{j,1},\ldots,\eps_{j,k} \in \F$, and to define
$$V_j = \{x\in\F^n \, :\, \ell_{j,i}(x) = \eps_{j,i}\, , \, \forall  i \in [k] \} \, .$$
Note that the uniform distribution has a lot of symmetries. Indeed, let $G$ be the group of affine transformations, generated by translations and $\mathbf{GL}_n(\F)$. Then for any $\gamma \in G$, $\Pu$ is invariant by action of $\gamma$ on $\F^n$. This rich symmetry structure yields a very precise description of random instances of $k$-{\sf FLAT} problems.

Under the planted distribution, (denoted by $\Pp$), an element $x^* \in\F^n$ is chosen uniformly. Conditioned on this element, the $V_j$s are independent and identically distributed, with a distribution denoted by $\Px$. Under this distribution, they are chosen uniformly on the set of flats of dimension $n-k$ that  do not contain $x^*$. They can be generated in a similar manner as under the uniform distribution, by drawing uniformly $k$ linearly independent linear forms $\ell_{j,1},\ldots,\ell_{j,k}$, and the $k$ values $\eps_{j,i}$ uniformly among the $2^k-1$ choices that are not all $\ell_{j,i}(x^*)$. We define $V_j$ similarly. By construction, it does not contain $x^*$, which is a satisfying assignment for $V$. \\

\begin{remark}
\label{REM:process} 
Let $G_{x^*}$ be the subgroup of $G$, the affine group consisting of affine transformations fixing $x^*$.  Then $G_{x^*}$ acts transitively on the $k$-flats not containing $x^*$.  In particular, a probability distribution on $k$-flats which is supported on $k$-flats not containing $x^*$, and which is invariant under $G_{x^*}$, must be uniform on the $k$-flats not containing $x^*$; in other words, it is the distribution $\Px$ described above.  In particular, the procedure of choosing $k$ linear forms $\ell_i$ and $k$ bits $\eps_i$ uniformly at random subject to the conditions that the $\ell_i$ are linearly independent, and that the $\ell_i(x^*) - \eps_i$ is nonzero for at least one $i$, is evidently $G$-invariant; thus, the resulting distribution on $k$-flats is $\Px$.  In this paper we will mostly use this description of $\Px$.  But we want to emphasize that there are many such descriptions, i.e. many distributions on $k$-tuples of pairs $(\ell,\eps)$ which yield the distribution $\Px$ on $k$-flats.  
\end{remark}
 
In order to avoid confusion regarding the representation of these flats, we consider here that the input data is the actual flat, given to us either as a membership oracle - a function that returns whether any element of $\F^n$ belongs to the flat $V_j$ - or as a uniformly random base $\ell_j$ of the space of linear forms that are constant on the flat, and the corresponding values $\eps_j$. From a purely statistical point of view, this makes no difference.

From an algorithmic point of view, we will consider that our data is a uniformly random basis of linear forms and the associated values $(\ell_j,\eps_j)$ for the $k$-flat, which has then the distribution above.

Formally, we denote by $q_0$ the uniform distribution on $k$-flats of in $\F^n$, and for all $x\in \F^n$ by $q_x$ the uniform distribution on $k$-flats of $\F^n$, that do not contain $x$. With these notations, the distributions considered in this problem are defined thus
\begin{equation*}
\Pu := q_0^{\otimes m}\; , \; \Pxsp := q_x^{\otimes m}\; , \; \Pp := \frac{1}{2^n} \sum_{x \in \F^n} \Pxs\, .
\end{equation*}
Our detection problem can be written as testing between two hypotheses
\begin{eqnarray*}
H_0 &:& V = (V_1,\ldots,V_m) \sim \Pu\\
H_0 &:& V = (V_1,\ldots,V_m) \sim \Pp \, .
\end{eqnarray*}

\section{Flat-satisfiability threshold}
\label{SEC:threshold}
In this section, we study the probability that a uniformly random instance $V$ of the $k$-{\sf FLAT} problem is flat satisfiable, when $m= \Delta n$, as a function of $\Delta>0$. This is achieved by studying the first two moments of $Z(V)$, number of satisfying assignments.
\begin{lemma}
\label{LEM:firstmoment}
Under the uniform distribution
$$\E[Z] = 2^n (1-2^{-k})^m \, .$$
\end{lemma}

Note that the first moment of $Z$ is the same when we consider the number of solutions in random $k$-{\sf SAT} formulas. Intrinsically, the group of symmetries $H$ of the uniform distribution for $k$-{\sf SAT} - generated by translations and permutations - and of the uniform distribution for $k$-{\sf FLAT} - the affine group $G$ - both act transitively on $\F^n$, which is the main point of the proof above. However, while the action of the affine group is also doubly transitive on $\F^n$, it is not the case for the action of $H$, which preserves Hamming distances for instance. This affects the computation of the second moments of $Z$, which is consequently very different under these two models.

\begin{lemma}
\label{LEM:secondmoment}
Let $V=(V_1,\ldots,V_m)$ be a random collection of $m$ $k$-flats on $\F^n$ with distribution $\Pu$. Let $m=\Delta n$, for some $\Delta >0$. We have
$$\frac{\E[Z^2]}{\E[Z]^2} \le 1+ o(1) + \frac{1}{\E[Z]} \, .$$
\end{lemma}

Together, Lemma~\ref{LEM:firstmoment} and \ref{LEM:secondmoment} yield the following
\begin{theorem}
\label{THM:sat}
For $k>0$ let $\Delta_k := \log(1/2)/\log(1-2^{-k})\approx 2^k \log(2)$. For $\Delta>0$, let $m = \Delta n$, and $V$ be uniformly distributed. When $m,n \rightarrow +\infty$, it holds that
\begin{itemize}
\item For $\Delta < \Delta_k$,  $\Pu(V \in {\sf FLAT}) \rightarrow 1$.
\item For $\Delta > \Delta_k$,  $\Pu(V \in {\sf FLAT}) \rightarrow 0$.
\end{itemize}
\end{theorem}

There is therefore a sharp phase transition in the linear regime, at $\Delta_k$, where the limit of the probability of flat satisfiability switches from 1 to 0. This result can be compared to the satisfiability transition for $k$-{\sf SAT} problems, for which $Z$ has the same expectation, but for which the second moment is much larger than $\E[Z]^2$. The proofs of satisfiability transitions \cite{AchPer04,CojPan13,DinSlySun14} are therefore much more technical, and this phenomenon does not occur at $\Delta_k$.

\section{Detection of planted flat-satisfiability}
\label{SEC:detect}
\subsection{Optimal rate}
\label{SEC:optimal}
One can understand the two distributions by the following generating process. Let $\cN_k$ be the number of subspaces of dimension $n-k$ in $\F^n$. There are therefore $2^k \cN_k$ possible $k$-flats (equivalent to a choice of linear forms, and $k$ values). Under the uniform distribution, $m$ flats are chosen independently and uniformly among the $2^k \cN_k$ possible choices. Under $\Px$, there is an excluded choice of values, and there are $(2^k-1) \cN_k$ allowed flats, among which we draw independently and uniformly $m$ flats. This interpretation of the distributions is useful to derive the likelihood ratio, in the following.

\begin{lemma}
\label{LEM:LR}
Let $V=(V_1,\ldots,V_m)$ be a collection of $m$ $k$-flats on $\F^n$,
$$\frac{\Pp}{\Pu}(V) = \frac{Z(V)}{\E[Z]}\, .$$
\end{lemma}

The distribution $\Pp$ therefore has a likelihood proportional to $Z(V)$: only the flat satisfiable $V$ have a positive measure, and those with a large number of flat satisfying assignments are more likely to occur. This can be contrasted with the uniform distribution on {\sf FLAT}, for which all flat satisfiable V are equally likely. One of the motivations behind the study of this likelihood ratio is its relationship with the total variation distance. Indeed, we have
$$d_{\sf TV}(\Pu,\Pp) = \frac{1}{2}\E\Big[\Big|\frac{Z}{\E[Z]}-1\Big|\Big] \le \frac{1}{2} \sqrt{\frac{\E[Z^2]}{\E[Z]^2}-1}\, .$$
The last inequality is a consequence of Jensen's inequality, and gives a more tractable bound on the total variation distance. It is equivalent to considering the $\chi^2$ divergence between the two distributions. When $\Delta <\Delta_k$, Lemma~\ref{LEM:secondmoment} yields
$$d_{\sf TV}(\Pu,\Pp)  \le \frac{1}{2}\sqrt{\frac{1}{\E[Z]}+o(1)} \rightarrow 0 \, .$$
Note that this approach is not fruitful to control the total variation distance in the $k$-{\sf SAT} planted satisfiability problem, as $\E[Z^2]$ is too large, in the linear regime of $m= \Delta n$ for some constant $\Delta>0$.

For this problem, when $\Delta >\Delta_k$, $\Pu(Z > 0) \le \E[Z] \rightarrow 0$. Checking flat satisfiability, i.e. if $Z>0$ is therefore a test with a one-sided probability of error equal to $\Pu(Z>0)$, as we have $\Pp(Z>0)=1$. Together, these two observations yield the following
\begin{theorem}
\label{THM:minimax}
For a fixed $\Delta>0$, let $m=\Delta n$. The following holds
\begin{itemize}
\item For $\Delta > \Delta_k$, and $\psi_{\sf FLAT}(V) = \mathbf{1}\{Z(V) >0\}$
$$\Pu(\psi_{\sf FLAT}=1) \vee \Pp(\psi_{\sf FLAT}=0) \rightarrow 0 \, .$$
\item For $\Delta < \Delta_k$, 
$$\inf_{\psi}\Pu(\psi=1) \vee \Pp(\psi=0) \rightarrow \frac 12 \, .$$
\end{itemize}
\end{theorem}

We observe in the statistical problem the same phase transition as in Theorem~\ref{THM:sat}: the problem switches at $\Delta_k$ from being insolvable (with a total variation distance converging to 0) to the existence of an powerful test, i.e. checking flat satisfiability. Note that in this regime, since $\E[Z]<1$, this test is equivalent to the likelihood ratio test $Z(V)>E[Z]$.

\subsection{Alternative planting distribution}
\label{SEC:alt}

The distribution $\Pp$ is a canonical way to draw a $k$-{\sf FLAT} instance that is surely satisfying while having independence of the $m$ $k$-flats, and having a simple distribution for each flat (conditionally on the choice of $x^*$). This is done in a similar spirit to the planted distribution used for $k$-{\sf SAT} instances \cite{Ber15,FelPerVem13}. More generally, let $\cP_{\sf FLAT}$ be the set of distributions on flat satisfiable instances defined as
$$P \in \cP_{\sf FLAT} \iff   \Pn(V \in {\sf FLAT}) = 1\, .$$
One can consider the more general problem of detecting planted flat satisfiability with the following hypothesis testing problem with an unknown planting distribution
\begin{eqnarray*}
H_0 &:& V = (V_1,\ldots,V_m) \sim \Pu\\
H_0 &:& V = (V_1,\ldots,V_m) \sim \Pone \in \cPF \, .
\end{eqnarray*}
The test $\psi_{\sf FLAT}$ exploits almost no property of the $\Pp$, apart from $\Pp(V \in {\sf FLAT}) = 1$, the sure existence of a flat satisfying assignment. Therefore, the upper bound described in Theorem~\ref{THM:minimax} would still hold for any choice of alternative distribution $\Pone \in \cPF$, and even for the composite hypothesis testing problem above. The lower bound is based on the fact that the likelihood ratio between $\Pp$ and $\Pu$ is equal to $Z/\E[Z]$, which is not true for all planting distribution $\Pone$. However, to prove a lower bound for the composite hypothesis testing problem, it suffices to obtain such a bound for one example of the set of distributions (here $\Pp \in \cPF$). Together these observations yield the following
\begin{theorem}
\label{THM:minimaxcomposite}
For a fixed $\Delta>0$, let $m=\Delta n$. The following holds
\begin{itemize}
\item For $\Delta > \Delta_k$, and $\psi_{\sf FLAT}(V) = \mathbf{1}\{Z(V) >0\}$
$$\Pu(\psi_{\sf FLAT}=1) \vee \sup_{\Pone \in \cPF}\Pone(\psi_{\sf FLAT}=0) \rightarrow 0 \, .$$
\item For $\Delta < \Delta_k$, 
$$\inf_{\psi} \big\{\Pu(\psi=1) \vee \sup_{\Pone \in \cPF}\Pone(\psi=0) \big\} \rightarrow \frac 12 \, .$$
\end{itemize}
\end{theorem}

Overall, the test $\psi_{\sf FLAT}$ is reliant on the fact that under the alternative, $V$ is satisfiable, not on {\em how} this satisfiability is achieved. We discuss further this feature of certain tests in Section~\ref{SEC:poly} and~\ref{SEC:hard}, when considering some algorithmic aspects of this decision problem.	

The picture is clear from the statistical and probabilistic point of view. However, from a computational point of view, checking if $Z$ is equal to 0 (i.e. if the union of flats covers $\F^n$) is an {\sf NP}-complete problem for $k \ge 3$, as $k$-{\sf SAT} is a particular case. An interesting question is whether there are detection methods that can solve this problem in an algorithmically efficient manner.

\section{Polynomial-time detection}
\label{SEC:poly}


 We study in this section the statistical performance of a test that runs in polynomial time. We introduce some notations necessary to define this test. Let $W$ be a $k$-flat of $\F^n$, defined by $k$ affine constraints
$$W = \{x\in\F^n \, :\, \ell_{i}(x) = \eps_{i}\, , \, \forall  i \in [k] \} \, .$$
We make the observation that $x$ does not lie on $W$ if and only if one of the above equations is not satisfied, or equivalently, taking $\alpha_i = 1-\eps_i$
$$x \notin W \iff P_{\ell, \alpha}(x) := \prod_{i=1}^k \big(\ell_i(x) + \alpha_i \big) =0 \, .$$
Factoring out, $P_{\ell,\alpha}$ can be written as a multivariate polynomial over $\F$ of degree $k$
$$P_{\ell, \alpha}(x) = \sum_{\substack{S \subset [n] \\ |S| \le k}} c_S(\ell,\alpha) \prod_{s \in S} x_s \, .$$
Note that all the monomials are squarefree, as $z^2=z$ for all $z\in \F$. Solving the $k$-{\sf FLAT} problem is therefore equivalent to solving a system of $m$ polynomial equations of degree $k$, an {\sf NP}-hard problem. In order to obtain a test that is computationally tractable, we lift this system of equations in a higher dimensional space to obtain a system of linear equations with quadratic constraints, that we will then relax. This general idea is common over reals \cite{Par01,Las01}, and adapted here in a finite field. In this particular context, this approach is inspired by \cite{AroGe11}, where this technique is used in a problem of learning with errors.

Let $N_k = \sum_{i=0}^k {n \choose i} \le (n+1)^k$, and for $x \in \F^n$, let $X \in \F^{N_k}$ such that $X_S = \prod_{s \in S} x_s$. We remark that $P_{\ell,\alpha}$ takes the same values as a linear form $\cL_{\ell ,\alpha}$ over $\F^{N_k}$, such that $P_{\ell,\alpha}(x) = \cL_{\ell ,\alpha}(X)$ for the $X$ associated to $x$, by taking
$$\cL_{\ell ,\alpha}(X) = \sum_{\substack{S \subset [n] \\ |S| \le k}} c_S(\ell,\alpha) X_S \, .$$
If we consider the mapping $\phi$ from $\F^n$ to $\F^{N_k}$, the so-called Veronese embedding, that associates $x$ to $X$, and $\cV \subset \F^{N_k}$ the image of $\phi$, it is equivalent to solve $P_{\ell,\alpha}(x) = 0$ over all of $\F^n$ and $\cL_{\ell ,\alpha}(X) = 0$ over $\cV$. In particular, determining if an instance of the $k$-{\sf FLAT} problem is flat satisfiable is equivalent to determining if a system of $m$ linear equations in $\F^{N^k}$ has a solution in $\cV$. The image $\cV$ can be written as the intersection of quadratic constraints of the type $X_{\{1\}} X_{\{2\}} = X_{\{1,2\}}$, making the system of equations intractable. In order to obtain a tractable approximation of this problem, we consider the relaxed linear system of equations, by keeping solely the constraint $X_{\emptyset} =1$. Formally, for an instance $V$ of the $k$-{\sf FLAT} problem, we will consider for each flat $V_j$ the associated linear form $\cL_{\ell_j,\alpha_j}$, and the overall system $\cL_V$ of $m+1$ linear equations in $\F^{N_k}$
\begin{equation}
\label{EQ:linsym}
\tag{$\cL_V$} \cL_{\ell_j ,\alpha_j}(X) = 0\; , \, \forall j \in [m]\; ; \; X_\emptyset =1 \, . 
\end{equation}
Note that if $x^* \in \F^n$ is flat-satisfiable for $V$, the associated $X^* = \phi(x^*) \in \F^{N_k}$ is a solution to $\cL_V$, as it is even a solution to the linear system of equations with stricter constraint $X \in \cV$. As a consequence, the system $\cL_V$ always has a solution for $V \sim \Pp$. However, under the uniform distribution, it is not always the case. 
\begin{lemma}
\label{LEM:syslin}
Recall that $\Delta_k := \log(1/2)/\log(1-2^{-k})\approx 2^k \log(2)$. Let $m = \Delta N_k$ for $\Delta>\Delta_k$, and $V=(V_1,\ldots,V_m) \sim \Pu$. The linear system $\cL_V$ has no solutions in $\F^{N_k}$, with probability converging to 1 when $n \rightarrow +\infty$.
\end{lemma}

We consider the test $\psi_{\cL}:V \mapsto \mathbf{1}\{\text{$\cL_V$ has a solution}\}$. When $m$ is of order $N_k \le (n+1)^k$, it is possible to construct and solve the linear system, and thus to determine the outcome of the test, in time $O(n^{3k})$, by Gaussian elimination. The result of Lemma~\ref{LEM:syslin} gives a guarantee, in terms of sample size, about the performance of this test.
\begin{theorem}
\label{THM:poly}
Let $m = \Delta n^k$, for $\Delta > \Delta_k$. It holds that
$$\Pu(\psi_\cL=1) \vee \Pp(\psi_\cL=0) \rightarrow 0 \, .$$
\end{theorem}

The test $\psi_\cL$ allows to distinguish the two distributions with probability of error going to 0, with computation time and sample size that are both polynomial in $n$. The statistical performance shown here is suboptimal, and it is not clear whether there exists a test that runs in time polynomial in $n$ and that can distinguish the two distributions with high probability for a sample size linear in $n$, the optimal regime. Perhaps the properties of the space of satisfiable assignment, such as the {\em shattering property} \cite{AchCoj08} could shed some light on these phenomena.

There are other detection problems for which the optimal regime of detection is not known to be attainable by algorithmically efficient testing methods. In particular, for the planted clique problem \cite{Jer92,Kuc95} in a graph with $n$ vertices, even though hidden cliques of size greater than $2\log_2(n)$ can be detected or recovered, polynomial-time algorithms are only known to be efficient at size  of order $\sqrt{n}$ \cite{AloKriSud98}, widely believed to be optimal. This hypothesis has recently been used as a primitive to show hardness for other learning problems. This problem, as well as those of estimating  planted assignments for CSP problems have been studied, and computational lower bounds shown to exist, in a specific computational model \cite{FelGriRey13,FelPerVem13}.

A common type of method to solve these detection problems, one that comes naturally to mind to find an improved algorithm for this problem - i.e.  that would need significantly less than $n^k$ samples - is to study the behavior of a judiciously chosen, tractable statistic $\sigma$ of the data $D$. When $D$ is constituted of $m$ independent samples, let us consider only $\sigma$ that are sums of statistics $\rho$ of $r$-tuples of the data, for a finite $r$. Simply, these approaches revolve around showing that $\sigma(D)$ behaves differently under the two distributions of interest, say $\E_{\text{uniform}}[\sigma(D)] =0$, and $\E_{\text{planted}}[\sigma(D)] =\mu >0$, and by showing that when the sample size is large enough, $\mu$ is much greater than the typical deviations of $\sigma$, making a test such as such as $\mathbf{1}\{\sigma(D) > \mu/2 \}$ powerful. Typical examples include statistics based on the degrees of vertices in a graph, bias in signs of literals in a CSP, etc. This type of approaches has been formalized in the notion of {\em statistical algorithms} \cite{FelGriRey13}, where instead of having access to i.i.d. samples $X_i$ with an unknown distribution, one has access to an oracle that returns, for any query function $f$, a value close to $\E[f(X)]$, up to some tolerance $\tau$. This generalizes the query model of \cite{Kea98}.

This is not the approach used here, where the test $\psi_\cL$ is based on the \textit{existence} of an element verifying certain properties - here being a solution to a linear system of equations in a finite field - not on summing a certain statistic over i.i.d samples (or couples, or triplets of these samples). This is a situation similar to the one described in Section~\ref{SEC:alt}, where the test $\psi_{\sf FLAT}$ relies solely on the fact that under the planted distribution, there exists a planted assignment.  Similarly, the result of Theorem~\ref{THM:poly} would still hold for any alternative distribution $\Pone \in \cPF$ or for the composite hypothesis testing problem on $\cPF$, as $V$ being flat satisfiable implies that $\cL_V$ has a solution.

In the following section, we describe a modified version of our hypothesis testing problem, by introducing the model of light planting. Even though it does not change the statistical nature of the problem, we show in Section~\ref{SEC:hard} it is as hard as the  ``Learning Parity with Noise'' problem, strongly suggesting that it cannot be efficiently solved. Therefore, it is highly improbable that any method that is robust to this modification - which is true for the approaches based on biases of statistics, as described above - could be successful for detection of planted flat satisfiability.

\section{Detection of lightly planted flat-satisfiability }
\label{SEC:light}
We consider a modified version of our hypothesis testing problem. It has the same null hypothesis and in the alternative, planting only happens with some constant probability $\pi \in (0,1)$, which we call light planting. This auxiliary problem is a useful tool to understand some computational aspects of our original decision problem (where $\pi=1$). Formally, we denote by $q_{x,\pi} :=(1-\pi)q_0 + \pi q_x$ the distribution on the flats of dimension $n-k$ that is mixture of the uniform $q_0$ and of  the planting distribution $q_x$, and define similarly $\Pxsp$ and $\Ppp$. As in the original planting model, we have
\begin{equation*}
\Pxsp := q_{x,\pi}^{\otimes m}\; , \; \Ppp := \frac{1}{2^n} \sum_{x \in \F^n} \Pxsp\, .
\end{equation*}
The alternative hypothesis is therefore replaced with $H_{1,\pi}: \,V = (V_1,\ldots,V_m) \sim \Ppp$, and this new detection problem is
\begin{eqnarray*}
H_0 &:& V = (V_1,\ldots,V_m) \sim \Pu\\
H_{1,\pi} &:& V = (V_1,\ldots,V_m) \sim \Ppp \, .
\end{eqnarray*}

This setting is different from problems with {\em quiet}, or {\em hidden} planting \citep[see, e.g.][]{KrzZde09}.
To tackle this problem, we consider for a given set of flats $V$ the following statistics 
$$s(V,x) = |\{j \, :\, x \notin V_j\}| \text{   , and    } \sigma(V) = \max_{x \in \F^n} s(V,x)\, .$$
They are respectively the number of flats of $V$ on which $x$ does not lie, and the maximum number of flat constraints simultaneously satisfiable by an element of $\F^n$. We derive the following deviation bounds for this second statistic under both hypotheses.
\begin{lemma}
\label{LEM:maxsat}
For a fixed $\Delta>0$, let $m = \Delta n$. It holds that
\begin{align*}
&\Pu \big(\sigma(V) > [(1-2^{-k}) + \alpha ] m \big) \le e^{-[2\alpha^2 \Delta -\log(2)] n}\\
&\Ppp\big(\sigma(V) < [(1-2^{-k}) + \pi 2^{-k} - \alpha ] m \big) \le e^{-2\alpha^2 \Delta n}\, .
\end{align*}

\end{lemma}

These deviation can be used to prove that a particular test is powerful in the linear regime.
\begin{theorem}
\label{THM:rateslight}
For a fixed $\Delta>0$, let $m = \Delta n$, $\tilde \Delta_{k,\pi}:=2^{2k-1} \log(2)/\pi^2$ and $\Delta_{k,\pi}:=2^{k} \log(2)/\pi^2$, and $\psi_\sigma(V) = \mathbf{1}\{\sigma(V) >[(1-2^{-k}) + \pi 2^{-(k+1)} ] m\}$. It holds that

For $\Delta > \tilde \Delta_{k,\pi}$, 
$\Pu(\psi_\sigma=1) \vee \Ppp(\psi_\sigma=0) \rightarrow 0$,\\

$\Delta < \Delta_{k,\pi}$,  $\inf_{\psi}\Pu(\psi=1) \vee \Ppp(\psi=0) \rightarrow \frac 12$.

\end{theorem}

If we consider $\pi$ to be a constant, the optimal rate of detection for the light planting version of the problem is therefore still in the linear regime $m=\Delta_{k,\pi} n$. Furthermore, the right dependency of $\Delta_{k,\pi}$ on $\pi$ is in $1/\pi^2$, up to constants that only depend on $k$.  

\section{Computational limits for planting detection}
\label{SEC:hard}
As noted above, the algorithmically efficient testing method $\psi_\cL$ described in Section~\ref{SEC:poly} can be used to solve this detection problem for any planting distribution in $\cP_{\sf FLAT}$, given a sample size of order $n^k$. It is however not robust to the modification of the hypothesis testing problem described in Section~\ref{SEC:light}: it relies heavily on the fact that for $V\sim \Pp$ (or any other planting distribution) there exists some $x^*$ that is flat-satisfiable, which guarantees in turn the existence of a solution to the linear system $\cL_V$. This reasoning does not go through under the light planting model. 

This phenomenon can be contrasted with the behavior of more standard testing methods, based on averages of simple statistics over samples, covered by the framework of statistical algorithms, or queries. Under this paradigm, testing methods are very sensitive to the choice of planting distribution (see, e.g. \cite{FelPerVem13} for a study of the effect of the planting distribution in CSPs on the sample complexity in estimation and detection problems), but not on the fact that the problem instance is actually satisfiable. Indeed, under the light-planting model, expectations under the alternative are only affected by a multiplicative constant $\pi$. 

We give here strong reasons to believe that improving the result of Theorem~\ref{THM:poly} - for the case $\pi=1$ - by using testing procedures of this type is hopeless, and provide a lower bound for statistical algorithms. Our reasoning is that such an approach would be robust to light planting, and would allow us to distinguish $\Pu$ and $\Ppp$ with sample size and running time polynomial in $n$. The following result shows that this would imply in turn the existence of an efficient method for the decision version of the  ``Learning Parity with Noise'' (LPN) problem of \cite{BluKalWas03}, known to be as hard as the recovery of the ``secret'' signal. This is conjectured to be a hard problem, for which the best algorithms run in time $2^{O(n/\log(n))}$, and used to prove the safety of cryptography systems (see \cite{Pie12}, and references within).

Let $(A,b) \in \F^{n \times m} \times \F^m$ be an instance of LPN. For each $j \in [m]$, let $\gamma_{j,1}, \ldots, \gamma_{j,k-1}$ be $k-1$ uniformly random, linearly independent linear forms on $\F^n$, themselves independent of the linear form $\varphi_j$ generated by $A_j$. If $A_j$ is uniformly random, the $n-k$ dimensional linear subspace of $\F^n$ that is the vanishing set of these $k$ linear forms is therefore uniformly random as well. Furthermore, let $\beta_{j,1}, \ldots, \beta_{j,k-1}$ be $k-1$ independent, uniformly random elements of $\F$, independent of $b_j$. Take $\ell_{j,1},\ldots, \ell_{j,k}$ be equal to $\gamma_{j,1}, \ldots, \gamma_{j,k-1},\varphi_j$ in a uniformly random order, and $\eps_{j,1}, \ldots, \eps_{j,k}$ be equal to $\beta_{j,1}, \ldots, \beta_{j,k-1},1-b_j$ in the same order. The equation $\ell_{j}(x) = \eps_j$ defines the $n-k$ dimensional flat $V_j$.

\begin{lemma}
\label{LEM:LPN}
Let $(A,b) \in \F^{n \times m} \times \F^m$, and $V$ the associated instance of $k$-{\sf FLAT} obtained by the procedure described above. The following holds
\begin{itemize}
\item If $(A,b)$ are independent and uniformly random, $V \sim \Pu$.

\item If $(A,b)$ is distributed as an instance of LPN with secret $x$, and probability of error $\eta <1/2$, $V \sim \Pxsp$, with $\pi =1-2\eta$.
\end{itemize}
\end{lemma}

\begin{remark}  Lemma~\ref{LEM:LPN} reduces the problem of distinguishing $\Pxsp$ from $\Pu$ to LPN.  The same argument reduces the problem of distinguishing $\Ppp$ from $\Pu$ to DLPN, the ``decision version" of LPN.  The DLPN problem, in turn, is at least as hard as LPN, by \cite[Theorem C.2]{AroGe11}.
\label{rem:DLPN}
\end{remark}

From a computational point of view, there is a very strong difference between the problems of detecting planted solutions to flat satisfiability, and detecting solutions that are only lightly planted, for any constant $\pi \in (0,1)$. It seems impossible to adapt the result of Theorem~\ref{THM:poly} to this new setting, and to describe an efficient algorithm that can distinguish these distributions for a sample size of order $n^k/\pi^2$, similarly to the result of Theorem~\ref{THM:rateslight}, or for any sample size that is polynomial in $n$. 

The testing methods based on simple statistics (i.e. sums of simpler statistics that depend on finite $r$-tuples of samples) as described in Section~\ref{SEC:poly}, are usually robust to these modifications. As an example, for the planted clique problem, consider a light planting distribution that only plants edges in the small subgraph with probability $\pi$. The sum of the degrees of all the vertices has mean $\frac{n(n-1)}{4}$ under the null, and respectively $\frac{n(n-1)}{4} + \frac{k(k-1)}{2}$ and $\frac{n(n-1)}{4} + \pi \frac{k(k-1)}{2}$ under the planted, or lightly planted distribution. Deviation bounds will therefore show that a test based on this statistic will be successful when $k \ge C \sqrt{n}$ under the planted model and $k \ge C \sqrt{n/\pi}$ under the lightly planted model, for some constant $C>0$. The rates of detection for this method are not  changed by this modification, for a constant $\pi$. The situation is similar for detection of planted satisfiability \cite[Thm 3.1]{Ber15}: a statistic based on joint signs of variables appearing several times in the formula has mean $0$ under the uniform distribution, and mean $1/[2(2^k-1)]$ under the planted distribution, and would have mean $\pi^2/[2(2^k-1)]$ under the light planting model. The necessary sample size $m$ of order $\sqrt{n}$ in this problem would only be affected in the constant by $\pi$.

This informal remark can be formalized within the setting of statistical algorithms, by the following
\begin{proposition}
Consider an hypothesis testing problem between distributions $q_0$ and $q_1$ that can be solved by $N$ queries of a statistical oracle with tolerance $\tau$. The hypothesis testing problem between $q_0$ and $q_{1,\pi} = (1-\pi) q_0 + \pi q_1$ can be solved by $N$ queries of a statistical oracle with tolerance $\tau \pi$.
\label{pr:statistical}
\end{proposition}

Indeed, for any bounded function $f$, it holds that $\E_{1,\pi} f - \E_0 f = \pi (\E_1 f - \E_0 f)$. As only the difference in expectation between these two distributions matter, it is equivalent to have access to an oracle with precision $\tau$ over either $q_0$ or $q_1$ or with precision $\pi \tau$ over $q_0$ or $q_{1,\pi}$. This is particularly important if this oracle is obtained by $m$ actual samples of the unknown distributions, in which case $\tau$ is of order $1/\sqrt{m}$. In this case, the necessary sample size needs only to be multiplied by a constant factor $1/\pi^2$ in order to obtain an oracle with the desired precision $\tau \pi$. Note that this propositions can be generalized to cases when the function $f$ is allowed to depend on a finite number of samples from the unknown distribution.

Proposition~\ref{pr:statistical} immediately implies that the $k$-{\sf FLAT} problem cannot be efficiently solved by a statistical oracle.

\begin{proposition}
No statistical oracle can be used to distinguish $\Pu$ from $\Pp$ in a number of queries polynomial in $n$.
\end{proposition}

\begin{proof}
By Proposition~\ref{pr:statistical}, a statistical oracle that could efficiently distinguish $\Pu$ from $\Pp$ could also efficiently distinguish $\Pu$ from $\Ppp$.  By Lemma~\ref{LEM:LPN} and Remark~\ref{rem:DLPN}, this is at least as hard as LPN.  In the computational model of statistical queries, it is known that an exponential number of queries are necessary to solve LPN (\cite{Kea98}), so no statistical algorithm can efficiently distinguish $\Pu$ from $\Pp$.
\end{proof}

As noted in sections \ref{SEC:alt} and \ref{SEC:poly}, the tests $\psi_{\sf FLAT}$ and $\psi_{\cL}$ studied for the problem of distinguishing $\Pu$ and $\Pp$ are robust to changes in the alternative distribution (i.e. the planting distribution), as long as it belongs to $\cPF$. They can even solve this problem when the planting distribution is unknown: this is the case of composite hypothesis testing. In this sense, they are able to refute, with high probability, most flat satisfiability instances when $m$ is greater than, respectively $\Delta_k n$ and $\Delta_k n^k$, while never refuting a flat satisfiable instance. This is reminiscent of a problem considered for 3-{\sf SAT} formulas by \cite{Fei02} in a hardness hypothesis. For the problem of satisfiability, the usual planted distribution does not illustrate well the hardness of this problem. Indeed, as mentioned above, there exists even a polynomial-time test that can distinguish the uniform and planted distribution with a sample size of order $\sqrt{n}$, which is optimal and well below the satisfiability threshold and the conjectured hard regime \cite{Ber15}. This test is also robust to the introduction of light planting, as it is based on distinguishing the expectation of a simple statistic over samples between the null and alternative hypotheses.

For the detection of planted flat-satisfiability, we show the existence of a test that can be decided in polynomial time, and that only necessitates a polynomial number of samples, and that never wrongly refutes a flat satisfiable instance (i.e. is powerful for all alternatives $\Pone \in \cPF$). However, as shown in Lemma~\ref{LEM:LPN}, these tests are not robust to other changes in the alternative, where planting yields instances that are {\em almost} flat satisfiable. An analogue of the problem, as considered in Hypothesis 2 in \cite{Fei02}, which weakens in this way the original hypothesis, would be as shown here, a much harder task.

\bibliographystyle{plainnat}
\bibliography{statebib2}
\newpage
\onecolumn{
\begin{center}
{\Large Supplementary material to\\ ``
Detection of Planted Solutions for Flat Satisfiability Problems''} 
\vskip 1.5 cm
\end{center}

\appendix
\section{Technical proofs}
\begin{proof}[of Lemma~\ref{LEM:firstmoment}]
It holds that
$$Z=\sum_{x\in \F^n} \prod_{i=1}^m \mathbf{1}\{x \notin V_j \} \, .$$
By linearity, symmetry of the distribution, and independence of the $V_j$, we have for any $x_0 \in \F^n$
$$\E[Z] = 2^n (\Pu(x_0 \notin V_1))^m \, .$$
Furthermore, for each $k$-flat of $\F^n$, $|V_1|=2^{n-k}$, which yields the desired result. 
\end{proof}

\begin{proof}[of Lemma~\ref{LEM:secondmoment} ]
We derive the second moment of $Z$
\begin{eqnarray*}
Z^2&=&\sum_{x,x'\in \F^n} \mathbf{1}\{x \in \cS(V) \}  \mathbf{1}\{x' \in \cS(V) \}\\
&=& \sum_{x}  \mathbf{1}\{x \in \cS(V) \} + \sum_{x\neq x'} \mathbf{1}\{x \in \cS(V) \} \mathbf{1}\{x' \in \cS(V) \}\, .
\end{eqnarray*}
Taking expectation yields
$$\E[Z^2] = \E[Z] + \sum_{x\neq x'} \Pu\big(\{x \in \cS(V) \} \cap \{ x' \in \cS(V) \} \big)\, .$$
The uniform distribution is invariant under the action of the affine group $G$, which is doubly transitive on $\F^n$. Therefore, the term $\Pu\big(\{x \in \cS(V) \} \cap \{ x' \in \cS(V) \} \big)$ is constant for all couples of distinct elements $(x,x')$ of $\F^n$. To compute this distribution, it thus suffices to consider that $x$ and $x'$ are uniformly randomly chosen among the set of pairs of distinct elements. For all $j \in [m]$, this yields
$$\Pu\big(\{x \notin V_j \} \cap \{ x' \notin V_j \} \big) = \frac{2^n - 2^{n-k}}{2^n} \cdot \frac{2^n - (2^{n-k} -1)}{2^n-1} = (1-2^{-k}) \Big(1-2^{-k} +\frac{2-2^{-k}}{2^n-1}\Big)\, .$$
Using this in the derivation of the second moment, we have
\begin{eqnarray*}
\E[Z^2]&=&\E[Z] + (2^{2n} -2^n) (1-2^{-k})^m \Big(1-2^{-k} +\frac{2-2^{-k}}{2^n-1}\Big)^m\\
&\le&  \E[Z] + 2^{2n}(1-2^{-k})^{2m} \Big(1+\frac{2-2^{-k}}{1-2^{-k}} \frac{1}{2^n-1}\Big)^m\\
&\le&  \E[Z] + \E[Z]^2 \Big(1 +\frac{2-2^{-k}}{1-2^{-k}} \frac{1}{2^n-1}\Big)^{\Delta n}\, .
\end{eqnarray*}
Note that the last term is a $1+o(1)$.
\end{proof}

 \begin{proof}[of Theorem~\ref{THM:sat} ]
We first note that $2 (1-2^{-k})^{\Delta_k} =1$, so that $\E[Z] = [2 (1-2^{-k})^{\Delta}]^n$ is exponentially large when $\Delta<\Delta_k$, and exponentially small when $\Delta>\Delta_k$.\\
\begin{itemize}
\item For $\Delta < \Delta_k$, Markov's inequality yields
$$\Pu(V \in {\sf FLAT}) = \Pu(Z(V) \ge 1) \le \E[Z] \rightarrow 0\,.$$

\item For $\Delta < \Delta_k$, Paley-Zigmund's inequality and the result of Lemma 3 yields
$$\Pu(V \in {\sf FLAT}) = \Pu(Z(V) > 0) \ge \frac{\E[Z]^2}{\E[Z^2]} \rightarrow 1\, .$$
\end{itemize}
\end{proof}

\begin{proof}[of Lemma~\ref{LEM:LR}]
By definition of $\Pp$
$$\frac{\Pp(V)}{\Pu(V) } = \frac{1}{2^n}\sum_{x \in \F^n} \frac{\Px(V)}{\Pu(V)}\, .$$
To compute the probabilities in the above ratios, we use the interpretation above of $m$ drawings in $N=2^k \cN_k$ possible flats independently if the distribution is $\Pu$, or otherwise in $N^*=(2^k-1)\cN_k$ possible choices corresponding to flats that do not contain $x^*$. Therefore, it holds for all $V$
\begin{equation*}
\frac{\Px(V)}{\Pu(V)} = \left\{
    \begin{array}{rl}
      0 &\; \text{if } x \notin \cS(V)\\
       \Big(\frac{N}{N^*}\Big)^m &\;\text{otherwise } 
    \end{array} \right.
\end{equation*}
Therefore, the likelihood ratio can be expressed in terms of $\mathbf{1}\{x \in \cS(V)\}$, and $N/N^* = 1/(1-2^{-k})$
\begin{eqnarray*}
\frac{\Pp}{\Pu}(V) &=& \frac{1}{2^n} \sum_{x \in \F^n} \Big(\frac{N}{N^*}\Big)^m \mathbf{1}\{x \in \cS(V)\}\\
 &=& \frac{1}{\E[Z]} \sum_{x \in \F^n} \mathbf{1}\{x \in \cS(V)\} = \frac{Z(V)}{\E[Z]}\, .
 \end{eqnarray*}
 \end{proof}

\begin{proof}[of Lemma~\ref{LEM:syslin} ]
Consider a fixed $Z \in \F^{N_k}$ such that $Z_\emptyset =1$. For an $k$-flat $W$ described by $(\ell,\alpha)$, we write $\cL_{\alpha,\ell}(Z)$ as a function $q_{Z,\ell}$ of $\alpha \in \F^k$
$$q_{Z,\ell}(\alpha) = \sum_{\substack{S \subset [n] \\ |S| \le k}} c_S(\ell,\alpha) Z_S\, .$$
We observe that each $c_S(\ell, \cdot)$ is a multivariate multilinear polynomial (with monomials that are squarefree), so that $q_{Z,\ell} \in \F[\alpha_1,\ldots,\alpha_k]$. Furthermore, the coefficient of the monomial $\alpha_1 \ldots \alpha_k$ is $Z_\emptyset=1$. As the squarefree monomials are linearly independent, there exists an element of $\F^k$ such that $q_{Z,\ell}(\alpha) \neq 0$. Therefore, as $\alpha$ is uniformly distributed under the uniform distribution $q_0$, it holds that
$$\Pu(\cL_{\alpha,\ell}(Z) = 0) = \Pu(q_{Z,\ell}(\alpha) = 0) \le 1-2^{-k}\, .$$
As an aside, note that this bound is tight. Indeed, for all $Z \in \cV$, the event $\cL_{\alpha,\ell}(Z) = 0$ is equivalent to $z\notin W$, for $z = \phi^{-1}(Z)$. The probability of this event is $1-2^{-k}$, as seen in the proof of Lemma~2.\\

Let $V=(V_1,\ldots,V_m) \sim \Pu$. By independence, we obtain directly that
$$\Pu(\cL_{\ell_j ,\alpha_j}(X) = 0\; , \, \forall j \in [m])  \le (1-2^{-k})^m\, .$$
By a union bound over all elements of $\F^{N_k}$, it holds that
$$\Pu(\text{$\cL_V$ has a solution}) \le 2^{N_k} (1-2^{-k})^m\, .$$
Taking $\Delta > \Delta_k$ yields the desired result.
\end{proof}

\begin{proof}[of Lemma~\ref{LEM:maxsat} ]
For all $x \in \F^n$, we observe that under the null hypothesis, the variable $s(x,V)$ has distribution $\cB(m,1-2^{-k})$. Therefore, by Hoeffding's inequality, 
$$\Pu \big(s(x,V) > [(1-2^{-k}) + \alpha ] m \big) \le \exp(-2\alpha^2 m)\,.$$
A union bound on $\F^n$ yields
$$\Pu \big(\sigma(V) > [(1-2^{-k}) + \alpha ] m \big) \le 2^n \exp(-2\alpha^2 m) \le \exp\big(- \big[2\alpha^2 \Delta -\log(2)\big] n\big)\,.$$

Under $\Px$ the variable $s(x^*,V)$ has distribution $\cB\big(m,(1-2^{-k}) + \pi 2^{-k}\big)$. By Hoeffding's inequality, 
$$\Pxp \big(s(x^*,V) < [(1-2^{-k})+\pi 2^{-k} -\alpha ] m \big) \le \exp(-2\alpha^2 m)\,.$$
By definition of $\Ppp$ and $\sigma(V) \ge s(x,V)$ for all $x \in \F^n$, we obtain the desired result.
\end{proof}

\begin{proof}[of Theorem 11~\ref{THM:rateslight} ]
For $\Delta> \tilde \Delta_{k,\pi}$, taking $\alpha = \pi 2^{-(k+1)}$ in the results of Lemma~10 yields the desired upper bound, as $2 \alpha^2 \Delta -\log(2)>0$.

For $\Delta<\Delta_{k,\pi}$, we derive a bound on the total variation distance $d_{\sf TV}(\Pu,\Ppp)$, through the inequality

$$d_{\sf TV}(\Pu,\Ppp) = \frac{1}{2}\E\Big[\Big|\frac{\Ppp}{\Pu}(V)-1\Big|\Big] \le \frac{1}{2} \sqrt{\E\Big[\Big(\frac{\Ppp}{\Pu}(V)-1\Big)^2\Big]}\, .$$
The term inside the square root being equal to the chi-square divergence $\chi^2(\Ppp,\Pu)$ between the two distributions. We write $\Pxsp = q_{x,\pi}^{\otimes m}$ and $\Pu=q_0^{\otimes m}$ as products of the distribution of each independent $V_j$. Writing out $\Ppp$ as a uniform mixture of the $\Pxsp$ yields
\begin{eqnarray*}
\chi^2(\Ppp,\Pu) &=& \frac{1}{2^{2n}} \sum_{x,x' \in \F^n} \E\Big[ \frac{\Pxsp}{\Pu} \frac{\Pxpp}{\Pu}(V)\Big] -1\\
&=& \frac{1}{2^{2n}}\sum_{x,x' \in \F^n} \E\Big[ \frac{q_{x,\pi}}{q_0} \frac{q_{x',\pi}}{q_0}(V_1)\Big]^m -1 \\
&=&\frac{1}{2^{2n}} \sum_{x \in \F^n} \E\Big[ \Big(\frac{q_{x,\pi}}{q_0} (V_1)\Big)^2\Big]^n + \frac{1}{2^{2n}}\sum_{x\neq x'} \E\Big[ \frac{q_{x,\pi}}{q_0} \frac{q_{x',\pi}}{q_0}(V_1)\Big]^m -1\, .
\end{eqnarray*}
Note that $q_{x,\pi} = (1-\pi) q_0 + \pi q_x$, where $q_x$ is the uniform distribution on $k$-flats that do not contain $x$ (the planting distribution), so that
$$\frac{q_{x,\pi}}{q_0} = 1+ \pi \Big[ \frac{q_{x}}{q_0}-1\Big]\, .$$
Substituting this in the above yields
\begin{eqnarray*}
\chi^2(\Ppp,\Pu) &=& \frac{1}{2^{2n}} \sum_{x \in \F^n} \Big(1+ \pi^2\Big[ \E\Big[\Big(\frac{q_{x}}{q_0}(V_1)\Big)^2\Big]-1\Big] \Big)^m \\
&&+ \frac{1}{2^{2n}}\sum_{x\neq x'} \Big(1+ \pi^2\Big[ \E\Big[\frac{q_{x}}{q_0}\frac{q_{x'}}{q_0}(V_1)\Big]-1\Big] \Big)^m -1 \, .
\end{eqnarray*}
Furthermore, for any $k$-flat $V_1$, it holds that $q_x/q_0(V_1) = (N/N_k) \mathbf{1}\{x \notin V_1\}$. We give the following upper bound the last two terms of this equation's RHS, 
\begin{eqnarray*}
\frac{1}{2^{2n}}\sum_{x\neq x'} \Big(1+ \pi^2\Big[ \E\Big[\frac{q_{x}}{q_0}\frac{q_{x'}}{q_0}(V_1)\Big]-1\Big] \Big)^m -1 &\le& \frac{1}{2^{2n}} 2^n \Big(1- \pi^2+ \pi^2\frac{\Pu(x,x' \notin V_1)}{(1-2^{-k})^2}\Big)^m-1 \\
&\le& \Big(\frac{1-\pi^2}{2} \Big)^n \Big(1+ \frac{\pi^2}{1-\pi^2} \frac{2-2^{-k}}{(1-2^{-k}} \frac{1}{2^n -1} \Big)^{\Delta n}-1\\
&\le& \Big(1+\frac{c_k \pi^2}{2^n -1} \Big)^{c_k n/\pi^2}-1\, ,
\end{eqnarray*}
for some constant $c_k>0$ (independent of $n$ and $\pi$), by the formula for $\Pu(x,x' \notin V_1)$ derived in the proof of Lemma~\ref{LEM:secondmoment}. The last term converges to 0 when $n \rightarrow +\infty$. We bound as well the first term of the main equation's RHS
\begin{eqnarray*}
\frac{1}{2^{2n}} \sum_{x \in \F^n} \Big(1+ \pi^2\Big[ \E\Big[\Big(\frac{q_{x}}{q_0}(V_1)\Big)^2\Big]-1\Big] \Big)^m &\le& \frac{1}{2^{2n}} 2^n (1 + \pi^2(\Pu(x \notin V_1)-1))^m \\
&\le& \frac{1}{2^n} \Big(1+ \frac{\pi^2}{2^k-1} \Big)^{\Delta n}\, .
\end{eqnarray*}
Taking $\Delta < \Delta_{k,\pi} = 2^k \log(2) /\pi^2$ yields $1/2 (1+ \pi^2/(2^k-1))^\Delta <1$, and all the terms of $\chi^2(\Ppp,\Pu)$ go to 0 when $n \rightarrow + \infty$.

\end{proof}

\begin{proof}[of Lemma 12~\ref{LEM:LPN} ]
In all cases, the $k$-flats are independent, and the $m$ sets of $k$ linear forms are uniformly distributed. If $(A,b)$ is uniformly random, so are the $b_j$, and as a consequence, the $\eps_j$. This yields the desired $V \sim \Pu$. However, if there is a secret $x$, $\phi_j(x) = 1-  b_j$ with probability $\eta$. The distribution of $1-b_j-\phi_j(x)$ is therefore is a mixture of the uniform distribution on $\F$ (with weight $1 -\pi$) and of the unit mass at $1$ (with weight $\pi$). The distribution of $\eps_j - \ell_j(x)$ is thus the mixture of the uniform distribution on $\F^n$ (with weight $1 -\pi$)
and of the the distribution on $\F^k \setminus\{0\}$ generated by placing a $1$ in one of the coefficients of $\eps_j - \ell_j(x)$, and letting the others be independent and uniform. As shown in Remark~1, the flat $V_j$ has distribution $q_{x,\pi}$ and $V\sim \Pxsp$, as desired.
\end{proof}
}

\end{document}